\documentclass[a4paper,reqno]{amsart}
\usepackage{geometry}
\usepackage{mathrsfs}
\usepackage{syntonly}
\usepackage{amsmath}
\usepackage{amsthm}
\usepackage{amsfonts}
\usepackage{amssymb}
\usepackage{latexsym}
\usepackage{amscd,amssymb,amsopn,amsmath,amsthm,graphics,amsfonts,mathrsfs,accents,enumerate,verbatim,calc}
\usepackage[dvips]{graphicx}
\usepackage[colorlinks=true,linkcolor=blue,citecolor=blue]{hyperref}
\input xy
\xyoption{all}

\usepackage{setspace}

\usepackage{color}

\numberwithin{equation}{section}

\theoremstyle{plain}

\newtheorem{thm}{Theorem}[section]
\newtheorem{cor}[thm]{Corollary}
\newtheorem{lem}[thm]{Lemma}
\newtheorem{prop}[thm]{Proposition}

\newtheorem{defn}[thm]{Definition}
\newtheorem{exm}[thm]{Example}

\newtheorem{rem}[thm]{Remark}

\newdir{ >}{{}*!/-5pt/\dir{>}}

\newcommand{\Hom}{\operatorname{Hom}\nolimits}

\newcommand{\op}{\operatorname{op}\nolimits}

\renewcommand{\mod}{\mathsf{mod}\hspace{.01in}}

\newcommand{\M}{\mathcal M}

\newcommand{\W}{\mathcal W}
\newcommand{\h}{\mathcal H}

\newcommand{\T}{\mathcal T}

\newcommand{\I}{\mathcal I}
\newcommand{\D}{\mathscr D}

\newcommand{\X}{\mathscr X}

\newcommand{\C}{\mathscr C}

\newcommand{\oT}{\overline{\T}}

\newcommand{\E}{\mathcal E}
\newcommand{\EE}{\mathbb E}

\newcommand{\svecv}[2]{\left(\begin{smallmatrix}
      #1 \\
      #2
    \end{smallmatrix}\right)}

\newcommand{\svech}[2]{\left(\begin{smallmatrix}
      #1 & #2
\end{smallmatrix}\right)}

\renewcommand{\emph}{\textit}
\renewcommand{\phi}{\varphi}

\newcommand{\add}{\mathsf{add}\hspace{.01in}}

\def\bfJ{\pmb{J}}

\def\catE{\mathcal{E}}
\def\catW{\mathcal{W}}
\def\modcat{\mathsf{mod}}
\def\obj{\mathrm{obj}}
\def\bbF{\mathbb{F}}
\newcommand{\kk}{\Bbbk}
\newcommand{\Q}{\mathcal{Q}} 
\def\I{\mathcal{I}}
\providecommand{\To}[1]{}
\renewcommand{\To}[1]{\mathop{-\!\!\!-\!\!\!\longrightarrow}\limits^{#1}}

\providecommand{\oT}[1]{}
\renewcommand{\oT}[1]{\mathop{\longleftarrow\!\!\!-\!\!\!-}\limits^{#1}}
\def\FF{\mathbb{F}}

\usepackage{tikz}
\usepackage{pgfplots}
\usepackage{tikz-3dplot}
\usetikzlibrary{patterns}
\usetikzlibrary{3d,calc}

\def\emb{\mathrm{emb}}
\def\epi{\mathrm{epi}}
\newcommand{\ind}{\mathsf{ind}}
\def\NN{\mathbb{N}}

\newcommand{\defines}[1]{{\it\color{blue}#1}}
\newcommand{\checkLYZ}[1]{{\it\color{red}#1}}

\begin{document}

\title{Quasi-abelian quotients in extriangulated categories}\footnote{\hspace{1em}
Yu-Zhe Liu is supported by the National Natural Science Foundation of China (Grant
No. 12401042) and Guizhou Provincial Basic Research Program (Natural Science) (Grant Nos. ZD[2025]085 and ZK[2024]YiBan066). Panyue Zhou is supported by the National Natural Science Foundation of China (Grant No. 12371034) and by the Scientific Research Fund of Hunan Provincial Education Department (Grant No. 24A0221).}

\author{Yu Liu, Yu-Zhe Liu and Panyue Zhou}

\address{ School of Mathematics and Statistics, Shaanxi Normal University, 710062 Xi'an, Shaanxi, P. R. China}
\email{recursive08@hotmail.com}

\address{School of Mathematics and Statistics, Guizhou University, 550025 Guiyang, Guizhou, P. R. China}
\email{liuyz@gzu.edu.cn}

\address{School of Mathematics and Statistics, Changsha University of Science and Technology, 410114 Changsha, Hunan, P. R. China}
\email{panyuezhou@163.com}


\begin{abstract}
Let $(\mathcal{E}, \mathbb{E}, \mathfrak{s})$ be an extriangulated category. Motivated by the theory of hereditary algebras, we introduce the notion of a \emph{hereditary-type subcategory} $\mathcal{W}\subseteq \mathcal{E}$. We prove that the quotient $\mathcal{E}/\mathcal{W}$ is a \emph{quasi-abelian category}, that is, an additive category with kernels and cokernels in which kernels are stable under pushouts and cokernels are stable under pullbacks. Moreover, we show that $\mathcal{E}/\mathcal{W}$ is abelian if and only if $\mathcal{W}$ is a \emph{cluster tilting subcategory} in a suitable relative extriangulated structure. Several examples are provided to illustrate the main results, showing that our approach both recovers known abelian hearts and yields new abelian or quasi-abelian quotients beyond classical settings.
\end{abstract}
\keywords{extriangulated category; quasi-abelian category; abelian quotient; relative cluster tilting subcategory
}
\subjclass[2020]{18E10; 18G80}
\maketitle

\section{Introduction}

Abelian categories \cite{F} form the classical foundation of homological algebra, since they guarantee the existence of kernels and cokernels and provide a well-behaved exact structure. This categorical framework allows homological methods to be applied systematically across algebra and geometry. Nevertheless, many important categories that arise in analysis, geometry, and representation theory are not abelian, even though they still retain enough structure to support homological arguments. To accommodate such situations, Schneiders introduced the notion of a quasi-abelian category \cite{Sc}.
A quasi-abelian category is an additive category with kernels and cokernels, where kernels remain kernels under pushouts and cokernels remain cokernels under pullbacks. This concept strictly extends the abelian framework: every abelian category is quasi-abelian, but not conversely. Classical examples include the category of Banach spaces with continuous linear maps, the category of locally convex topological vector spaces, and categories of filtered modules. These examples show that quasi-abelian categories provide a natural environment for homological methods in contexts where abelian categories are too restrictive.

Nakaoka and Palu \cite{NP} introduced extriangulated categories, which unify exact and triangulated categories in a single framework. Extriangulated categories have proved to be a flexible tool in representation theory and related areas, allowing homological phenomena from both exact and triangulated contexts to be studied simultaneously. A natural question is how abelian or quasi-abelian structures can arise as quotients of extriangulated categories, and how cluster tilting theory interacts with these quotients.
The study of abelian quotients from extriangulated categories has a rich history. Nakaoka \cite{N1} proved that the heart of a cotorsion pair in a triangulated category is an abelian category.
Liu \cite{L1} showed that the heart of a cotorsion pair in an exact category is an abelian category.
Liu and Nakaoka \cite{LN} further developed the theory for extriangulated categories, showing that the heart of a twin cotorsion pair is always semi-abelian. Subsequently, Hassoun and Shah \cite{HS} proved that under certain conditions, these hearts are quasi-abelian. More recently, Liu, Yang and Zhou \cite{LYZ} showed that any heart of a twin cotorsion pair admits a largest exact category structure and is always quasi-abelian. These results provide a systematic framework for understanding when quotients of extriangulated categories carry algebraic structure.

On the other hand, cluster tilting theory provides a powerful mechanism for constructing abelian categories from triangulated and exact categories. The fundamental observation, due to Koenig and Zhu \cite{KZ}, is that for a triangulated category $\C$ with a cluster tilting subcategory $\T$, the quotient category $\C/\T$ carries a natural abelian structure. More precisely, Koenig and Zhu showed that the quotient functor induces an equivalence $\C/\T\simeq \mod \T$, where $\mod \T$ denotes the category of finitely presented contravariant additive functors from $\T$ to the category of abelian groups.
Demonet and Liu \cite{DL} proved that some subquotient categories of exact categories are abelian. This generalizes a result of Koenig-Zhu for triangulated categories. As a particular case, if an exact category $\mathcal{B}$ with enough projective and injective objects admits a cluster tilting subcategory $\mathcal{M}$, then the quotient category $\mathcal{B}/\mathcal{M}$ is abelian. More precisely, $\mathcal{B}/\mathcal{M}$ is equivalent to the category of finitely presented modules over $\underline{\mathcal M}$.
These results reveal a deep connection between the combinatorics of cluster tilting and the homological algebra of abelian categories, providing a conceptual framework that unifies the construction of module categories from both representation-theoretic and categorical perspectives.

On the other hand, cluster tilting theory offers a natural approach to constructing abelian categories from triangulated or exact categories. A fundamental result due to Koenig and Zhu \cite{KZ} states that if a triangulated category $\C$ admits a cluster tilting subcategory $\T$, then the quotient category $\C/\T$ is abelian. Moreover, the quotient functor induces an equivalence
$\C/\T \simeq \mod \T$, where $\mod \T$ denotes the category of finitely presented contravariant additive functors from $\T$ to the category of abelian groups.
This construction was subsequently extended to the exact setting by Demonet and Liu \cite{DL}, who proved that certain subquotient categories of exact categories are abelian. In particular, if an exact category $\mathcal B$ with enough projective and injective objects admits a cluster tilting subcategory $\mathcal M$, then the quotient category $\mathcal B/\mathcal M$ is abelian. Furthermore, they established an equivalence
$\mathcal B/\mathcal M \simeq \mod \underline{\mathcal M},$
where $\underline{\mathcal M}$ denotes the stable category of $\mathcal M$.
These results indicate that cluster tilting subcategories serve as a bridge between triangulated or exact categories and abelian categories. They provide a unified framework for realizing module categories as quotient categories and reveal a close interaction between cluster tilting theory, homological algebra, and the representation theory of finite-dimensional algebras.

The study of abelian quotient categories arising from cluster tilting subcategories has subsequently been extended in several directions. Beligiannis \cite{Be} developed a theory of triangulated subfactors and abelian localizations, showing that suitable rigid subcategories of triangulated categories give rise to abelian quotient categories. Zhou and Zhu \cite{ZZ} further investigated triangulated quotient categories
 in extriangulated categories and established equivalences between certain quotient categories and module categories over stable categories. These developments suggest that the existence of abelian quotient categories is governed by more general homological structures than cluster tilting subcategories alone, thereby motivating the study of analogous constructions associated with cotorsion pairs and related categorical structures.

The preceding developments naturally raise the question of whether similar constructions can be obtained in the broader framework of extriangulated categories. Introduced by Nakaoka and Palu \cite{NP}, extriangulated categories simultaneously generalize exact categories and triangulated categories, providing a common setting for a wide range of homological phenomena. In this framework, Nakaoka \cite{N2} established a general heart construction associated with cotorsion pairs and proved that the resulting heart is abelian. This construction unifies the classical hearts arising from both $t$-structures and cluster tilting subcategories in triangulated categories. Subsequently, Liu and Nakaoka \cite{LN} extended the theory to twin cotorsion pairs in extriangulated categories and showed that the associated heart is always semi-abelian. These results suggest that the appearance of abelian and semi-abelian categories is closely related to cotorsion-theoretic structures in extriangulated categories and provide further evidence that quotient and heart constructions admit a common homological interpretation in this broader setting.

In this paper we address this problem by introducing the notion of a {\bf hereditary-like} subcategory in an extriangulated category.
\begin{defn}
Let $(\E,\EE,\mathfrak{s})$ be an extriangulated category. A subcategory $\W\subsetneq \E$ is called \emph{hereditary-like} if the following conditions are satisfied:
\begin{itemize}
  \item[(1)] $\W$ is closed under direct sums and direct summands;
  \item[(2)] For any object $A\in \E$, there exists an $\EE$-triangle $A \xrightarrow{w^A} W^1\to W^2\dashrightarrow$ where $W^1,W^2\in \W$ and $w_A$ is a left $\W$-approximation;
  \item[(3)] For any object $A\in \E$, there exists an $\EE$-triangle $W_2\to W_1 \xrightarrow{w_A} A\dashrightarrow$ where $W^1,W^2\in \W$ and $w_A$ is a right $\W$-approximation.
\end{itemize}
\end{defn}

This notion is motivated by the theory of hereditary algebras. For example, if $kQ$ is a finite dimensional tame hereditary algebra, the subcategory $\W = \add({\rm Proj} ~kQ\oplus {\rm Inj} ~kQ)$ is hereditary-like in the module category $\mod kQ$ (see Example \ref{ex2} for details). 

For such a subcategory $\W$, we prove the following result.

\begin{thm}{\rm (see Theorem \ref{main} for details)}
Let $(\E,\EE,\mathfrak{s})$ be an extriangulated category and $\W$ be a hereditary-like subcategory. Then $\E/\W$ is quasi-abelian
\end{thm}

Furthermore, we establish that $\E/\W$ is abelian if and only if $\W$ is a cluster tilting subcategory in a suitable relative extriangulated structure under certain assumptions.

\begin{thm} {\rm (see Theorem \ref{thm3} for details)}
Let $(\E,\EE,\mathfrak{s})$ be a Krull-Schmidt extriangulated category and $\W$ be a hereditary-like subcategory. Assume that $\E/\W$ is left locally finite. Then $\E/\W$ is abelian if and only if $\W$ is an $\EE_{\mathfrak{e}}$-cluster tilting subcategory in extriangulated subcategory $(\E,\EE_{\mathfrak{e}},\mathfrak{s}|_{\EE_{\mathfrak{e}}})$.
\end{thm}

These results reveal a close relationship between cluster tilting theory and abelian quotient categories, extending several classical constructions from exact and triangulated categories to the extriangulated setting. At the same time, extriangulated categories provide a common framework encompassing both exact and triangulated categories, while relative extriangulated structures offer a natural setting for the formulation and study of cluster tilting conditions.
\vspace{1mm}

The paper is organized as follows. In Section 2, we review some preliminary material on extriangulated categories. Section 3 introduces hereditary-like subcategories and proves that the quotient category $\E/\W$ is quasi-abelian. In Section 4, we establish the relationship between cluster tilting subcategories and abelian quotient categories. Section 5 concludes the paper with several examples illustrating the main results.

\section{Preliminaries}

We first recall the terminology and basic properties of extriangulated categories that will be used throughout the paper. For the complete definitions and foundational results, we refer the reader to \cite[Sections 2, 3]{NP}.

Let $\E$ be an additive category, and let
\[
\mathbb{E} \colon \E^{\mathrm{op}} \times \E \longrightarrow \mathrm{Ab}
\]
be an additive bifunctor, where $\mathrm{Ab}$ denotes the category of abelian groups. For objects $A,C\in\E$, an element $\delta\in\mathbb{E}(C,A)$ is called an \emph{$\mathbb{E}$-extension}.

Let $\mathfrak{s}$ be a correspondence which assigns to each $\mathbb{E}$-extension $\delta\in\mathbb{E}(C,A)$ an equivalence class
\[
\mathfrak{s}(\delta)
=
\xymatrix@C=0.8cm{[A \ar[r]^x & B \ar[r]^y & C]}.
\]
Such a correspondence $\mathfrak{s}$ is called a \emph{realization} of $\mathbb{E}$ if it satisfies the compatibility conditions in~\cite[Definition~2.9]{NP}.

A triple $(\mathcal{E},\mathbb{E},\mathfrak{s})$ is called an \emph{extriangulated category} if the following conditions hold:
\begin{enumerate}
  \item $\mathbb{E} \colon \mathcal{E}^{\mathrm{op}}\times\mathcal{E}\to\mathrm{Ab}$ is an additive bifunctor;
  \item $\mathfrak{s}$ is an additive realization of $\mathbb{E}$;
  \item the pair $(\mathbb{E},\mathfrak{s})$ satisfies the axioms in~\cite[Definition~2.12]{NP}.
\end{enumerate}
\vspace{2mm}

When no confusion can arise, we simply write $\mathcal{E}$ for the extriangulated category $(\mathcal{E},\mathbb{E},\mathfrak{s})$.   Throughout this article, unless otherwise stated, we assume that $\E:=(\mathcal{E},\mathbb{E},\mathfrak{s})$ is an extriangulated category. Whenever we say that $\mathcal{D}$ is a subcategory of $\mathcal{E}$, we always mean that $\mathcal{D}$ is full and closed under isomorphisms.
We will use the following terminology.

\begin{defn}[{\cite[Definitions~2.15 and~2.19]{NP}}]
Let $\mathcal{E}$ be an extriangulated category.
\begin{enumerate}
  \item[{\rm (1)}] A sequence
  \[
  A \xrightarrow{x} B \xrightarrow{y} C
  \]
  is called a \emph{conflation} if it realizes an $\mathbb{E}$-extension $\delta\in\mathbb{E}(C,A)$. In this case, $x$ is called an \emph{inflation}, and $y$ is called a \emph{deflation}.

  \item[{\rm (2)}] If a conflation
  \[
  A \xrightarrow{x} B \xrightarrow{y} C
  \]
  realizes an extension $\delta\in\mathbb{E}(C,A)$, then the pair
  \[
  (A \xrightarrow{x} B \xrightarrow{y} C,\delta)
  \]
  is called an \emph{$\mathbb{E}$-triangle}. We write it as
  \[
  A \overset{x}{\longrightarrow} B
  \overset{y}{\longrightarrow} C
  \overset{\delta}{\dashrightarrow}.
  \]
  If the extension class $\delta$ is clear or irrelevant, we omit it from the notation.

  \item[{\rm (3)}] Let
  \[
  A \overset{x}{\longrightarrow} B
  \overset{y}{\longrightarrow} C
  \overset{\delta}{\dashrightarrow}
\hspace{2mm}\mbox{and}\hspace{2mm}
  A' \overset{x'}{\longrightarrow} B'
  \overset{y'}{\longrightarrow} C'
  \overset{\delta'}{\dashrightarrow}
  \]
  be two $\mathbb{E}$-triangles. A triple of morphisms $(a,b,c)$ is called a \emph{morphism of $\mathbb{E}$-triangles} if it realizes a morphism $(a,c)\colon\delta\to\delta'$. In this case, we have a commutative diagram
  \[
  \xymatrix{
  A \ar[r]^x \ar[d]^a
    & B \ar[r]^y \ar[d]^b
    & C \ar@{-->}[r]^{\delta} \ar[d]^c
    & {} \\
  A' \ar[r]^{x'}
    & B' \ar[r]^{y'}
    & C' \ar@{-->}[r]^{\delta'}
    & {} .
  }
  \]
\end{enumerate}
\end{defn}

The following lemma will be used frequently in what follows.

\begin{lem}[{\cite[Proposition 1.20]{LN}\label{L4}}]
Let $A\xrightarrow{x} B\xrightarrow{y} C\overset{\delta}\dashrightarrow$ be any $\EE$-triangle. Let $f:A\to D$ be any morphism and $D \xrightarrow{d} E\xrightarrow{e} C\dashrightarrow$ be an $\EE$-triangle realizing $f_*\delta$. Then there is a morphism $g:B\to E$ which gives a commutative diagram of $\EE$-triangles
$$\xymatrix{
A\ar[r]^x \ar[d]_{f} &B\ar[r]^y \ar[d]^{\exists ~g} &C\ar@{-->}[r]^{\delta} \ar@{=}[d]&\\
D \ar[r]_d &E\ar[r]_e &C\ar@{-->}[r]_{f_*\delta} &
}
$$
and an $\EE$-triangle $A \xrightarrow{\svecv{-f}{x}} D\oplus B \xrightarrow{\svech{d}{g}} E \overset{e^*\delta}\dashrightarrow$. This means the left square of the commutative diagram is a weak push-out and  a weak pull-back:
\begin{itemize}
\item[(i)] If there are morphisms $s:D\to R$ and $t:B\to R$ such that $sf=tx$, then there exists a morphism $r:E\to R$ (not necessarily unique) which makes the following diagram commute:
$$\xymatrix{
A\ar[r]^x \ar[d]_{f} &B \ar[d]^g \ar@/^8pt/[ddr]^t\\
D \ar[r]^d \ar@/_8pt/[drr]_s &E \ar@{.>}[dr]^r\\
&&R
}
$$
\item[(ii)] If there are morphisms $s':R'\to D$ and $t':R'\to B$ such that $ds'=gt'$, then there exists a morphism $r':R'\to A$ (not necessarily unique) which makes the following diagram commute:
$$\xymatrix{
R' \ar@/_8pt/[ddr]_{s'} \ar@/^8pt/[drr]^{t'} \ar@{.>}[dr]^{r'}\\
&A\ar[r]^x \ar[d]_{f} &B \ar[d]^g\\
&D \ar[r]^d  &E
}
$$
\end{itemize}
\end{lem}

\begin{defn}\label{DF1} \rm
Let $\C,\D$ be subcategories of $\E$.
For any pair of objects $Z,X\in \E$, let $\EE^{\D}(Z,X)$ be the subset of $\EE(Z,X)$ such that $\delta\in \EE^{\D}(Z,X)$ if it admits an $\EE$-triangle
$$\xymatrix{X\ar[r]^x &Y\ar[r]^y &Z\ar@{-->}[r]^{\delta} &}$$
which satisfies the following condition:
\begin{itemize}
\item[(m)] $x$ is $\D$-monic, which means $\Hom_{\E}(Y,D)\xrightarrow{\Hom_{\E}(x,D)}\Hom_{\E}(X,D)$ is surjective for any $D\in \D$.
\end{itemize}

For convenience, such $\EE$-triangles are called \defines{$\EE^{\D}$-triangles}.
By \cite[Definition 3.18 and Proposition 3.19]{HLN} for $n = 1$,
$\EE^{\D}$ is a closed subfunctor of $\EE$.
\vspace{1mm}

Dually, let $\EE_{\C}(Z,X)$ be the subset of $\EE(Z,X)$ such that $\sigma\in \EE_{\C}(Z,X)$ if it admits an $\EE$-triangle
$$\xymatrix{X\ar[r]^x &Y\ar[r]^y &Z\ar@{-->}[r]^{\sigma} &}$$
which satisfies the following condition:
\begin{itemize}
\item[(e)] $y$ is $\C$-epic, which means $\Hom_{\E}(C,Y)\xrightarrow{\Hom_{\E}(C,y)}\Hom_{\E}(C,Z)$ is surjective for any $C\in \C$.
\end{itemize}

For convenience, such $\EE$-triangles are called \defines{$\EE_{\C}$-triangles}. By  \cite[Definition 3.18 and Proposition 3.19]{HLN} for $n = 1$, $\EE_{\C}$ is also a closed subfunctor of $\EE$.
\vspace{1mm}

Finally, let $\EE_{\C}^{\D}(Z,X)$ be the subset of $\EE(Z,X)$, consisting of those $\eta\in \EE_{\C}(Z,X)$ which admits an $\EE$-triangle
$$\xymatrix{X\ar[r]^x &Y\ar[r]^y &Z\ar@{-->}[r]^{\eta} &}$$
which satisfies condition {\rm (m)} and {\rm (e)}. For convenience, such $\EE$-triangles are called \defines{$\EE_{\C}^{\D}$-triangles}. $\EE^{\D}_{\C}$ is a closed subfunctor of $\EE$.

\end{defn}

Let $\mathbb{F}$ be a subfunctor of $\EE$. Denote by $\mathfrak{s}|_{\mathbb{F}}$ the restriction of $\mathfrak{s}$ onto the $\EE$-extensions realized by $\mathbb{F}$-triangles. We have the following proposition.

\begin{prop}\label{main1}
Let $\C,\D$ be a subcategory of $\E$. Then the following
\begin{itemize}
\item[(1)] $(\E,\EE_{\D},\mathfrak{s}|_{\EE_{\D}})$,
\item[(2)] $(\E,\EE^{\C},\mathfrak{s}|_{\EE^{\C}})$,
\item[(3)] $(\E,\EE^{\C}_{\D},\mathfrak{s}|_{\EE^{\C}_{\D}})$
\end{itemize}
are extriangulated subcategory of $(\E,\EE,\mathfrak{s})$. Moreover, if $\C=\D$, then $\D$ is a subcategory of projective-injective objects in $(\E,\EE^{\D}_{\D},\mathfrak{s}|_{\EE^{\D}_{\D}})$, $\E/\D$ inherits an extriangulated category structure from $(\E,\EE^{\D}_{\D},\mathfrak{s}|_{\EE^{\D}_{\D}})$.
\end{prop}

\begin{defn}\rm
A subcategory $\W\subsetneq \E$ is called \defines{hereditary-like} if the following conditions are satisfied:
\begin{itemize}
  \item[(1)] $\W$ is closed under direct sums and direct summands.
  \item[(2)] For any object $A\in \E$, there exists an $\EE$-triangle $A \xrightarrow{w^A} W^1\to W^2\dashrightarrow$ where $W^1,W^2\in \W$ and $w_A$ is a left $\W$-approximation.
  \item[(3)] For any object $A\in \E$, there exists an $\EE$-triangle $W_2\to W_1 \xrightarrow{w_A} A\dashrightarrow$ where $W^1,W^2\in \W$ and $w_A$ is a right $\W$-approximation.
\end{itemize}
\end{defn}

\begin{exm}\rm
Let $k$ be a field and $Q$ be the following quiver:
$$1\to 2\to \cdots \to n.$$
Let $\W$ be subcategory consisted by the direct sums of the projective modules and injective modules in $\mod kQ$. Then $\W$ is a hereditary-like subcategory in $\mod kQ$. Moreover, $\mod kQ/\W$ is abelian.
\end{exm}

\begin{exm}\label{ex2}\rm
Let $k$ be a field and consider the finite dimensional path algebra $kQ$, where $Q$ is a quiver of tame type. Let $\W$ be subcategory consisted by the direct sums of the projective modules and injective modules in $\mod kQ$. If $kQ$ is a finite dimensional hereditary algebra, then $\W$ is a hereditary-like subcategory in $\mod kQ$.
\end{exm}

From now on, we always assume that subcategory $\W$ is hereditary-like.

\begin{lem}
If $\W$ is hereditary-like, then it contains all the projective objects and all the injective objects.
\end{lem}

\begin{proof}
For any projective object $P$, there exists an $\EE$-triangle
$$P \xrightarrow{~w^P~} W^1\to W^2\dashrightarrow$$
 where $W^1,W^2\in \W$. Since $P$ is projective, the above $\EE$-triangle
 splits, thus $w^P$ is \checkLYZ{a} section, that is, $P$ is a direct summand of $W^1\in \W$.
 Since $\W$ is closed under direct summands, we have $P\in\W$.

 Similarly, we can show that $\W$ contains all the injective objects.
\end{proof}

We recall some notations.
For a subcategory $\W \subseteq \E$, we denote by
\[
{^{\bot_1}}\W = \{\, X \in \E \mid \mathbb{E}(X,W)=0 \ \text{for all } W \in \W \,\};
\]
\[
\W^{\bot_1} = \{\, Y \in \E \mid \mathbb{E}(W,Y)=0 \ \text{for all } W \in \W \,\}.
\]

\begin{lem}
If $\W$ is hereditary-like, then ${^{\bot_1}}\W\subseteq \W$ and $\W^{\bot_1}\subseteq \W$.
\end{lem}

\begin{proof}
 For any $M\in{^{\bot_1}}\W$, we have $\mathbb{E}(M,\W)=0$.
Since  $\W$ is hereditary-like, then there exists an  $\EE$-triangle
 $$W_2\to W_1 \xrightarrow{~w_M~} M\dashrightarrow$$
  where $W^1,W^2\in \W$. Since $\mathbb{E}(M,W_2)=0$,
  we have that the above $\EE$-triangle splits. It follows that
  $w_M$ is retraction and then $M$ is a direct summand of $W_1\in W$.
  Since $\W$ is closed under direct summands, we have $M\in\W$.

Similarly, we can show $\W^{\bot_1}\subseteq \W$.
\end{proof}

\begin{prop}\label{twin}
Let $\E$ be a Krull--Schmidt extriangulated category with enough projectives and injectives.
Suppose that $\W$ is hereditary-like and closed under extensions.
Then
$$
\big(({^{\bot_1}}\W,\W),\,(\W,\W^{\bot_1})\big)
$$
forms a twin cotorsion pair in $\E$.
\end{prop}

\begin{proof}
Since $\W$ is hereditary-like, we have ${^{\bot_1}}\W\subseteq \W$ and $\W^{\bot_1}\subseteq \W$. We show that $(\W,\W^{\bot_1})$ is a cotorsion pair, then dually $({^{\bot_1}}\W,\W)$ is also a cotorsion pair. For any object $A\in \E$, there is an $\EE$-triangle $W_1\to W_0\xrightarrow{w} A\dashrightarrow$ where $W_1,W_0\in \W$ and $w$ is a right $\W$-approximation. Since $\E$ is Krull-Schmidt, we can assume that $w$ is right minimal. Then by Wakamatsu's Lemma, $W_1\in \W^{\bot_1}$. Since $\E$ has enough injectives, $A$ admits an $\EE$-triangle $A\to I\to A_1\dashrightarrow$ where $I$ is injective. $A_1$ admits an $\EE$-triangle $W_1'\to W_0'\xrightarrow{w'} A_1\dashrightarrow$ where $W_1',W_0'\in \W$. We can also assume that $W_1'\in \W^{\bot_1}$. Then we have the following commutative diagram
$$\xymatrix{
&W_1' \ar[d] \ar@{=}[r] &W_1' \ar[d]\\
A \ar[r] \ar@{=}[d] &W_1'\oplus I \ar[r] \ar[d] &W_0' \ar[d] \ar@{-->}[r] &\\
A \ar[r] &I \ar@{-->}[d] \ar[r] &A_1 \ar@{-->}[d] \ar@{-->}[r] &\\
&&
}
$$
where $W_1'\oplus I\in \W^{\bot_1}$. Hence  $(\W,\W^{\bot_1})$ is a cotorsion pair.
\end{proof}

\vspace{2mm}

Let $A\xrightarrow{\underline a} B$ be an epimorphism in $\E/\W$. Fix an $\EE$-triangle $A\xrightarrow{w_A} W^1 \to W^2 \dashrightarrow$
where $W^1,W^2\in \W$ and $w_A$ is a left $\W$-approximation, and an $\EE$-triangle $B\xrightarrow{w_B} W^3 \to W^4 \dashrightarrow$
where $W^3,W^4\in \W$ and $w_B$ is a left $\W$-approximation. Then we obtain the following commutative diagram of $\EE$-triangles
$$\xymatrix{
A \ar[r]^{w_A} \ar[d]_a &W^1 \ar[r] \ar[d]^{w} &W^2 \ar@{=}[d] \ar@{-->}[r] &\\
B \ar[r]^b \ar@{=}[d] &C \ar[r] \ar[d] &W^2 \ar[d] \ar@{-->}[r] &\\
B \ar[r]_{w_B} &W^3 \ar[r]  &W^4 \ar@{-->}[r]&.
}
$$
We have the following lemma.

\begin{lem}\label{L2}
If $\E$ is Krull-Schmidt, then $C\in \W$.
\end{lem}

\begin{proof}
Since $a$ is an epimorphism in $\E/\W$ and $\underline {ba}=0$, $\underline b=0$, which implies that $b$ factors through a object $W^0\in \W$. Since $w_B$ is a left $\W$-approximation, we have the following commutative diagram
$$\xymatrix{
B \ar[r]^b \ar@{=}[d] &C \ar[r] \ar[d] &W^2 \ar[d] \ar@{-->}[r] &\\
B \ar[r]^{w_B} \ar@{=}[d] &W^3 \ar[r] \ar[d]  &W^4 \ar@{-->}[r] \ar[d] &\\
B \ar[r]^b &C \ar[r]  &W^2 \ar@{-->}[r] &.
}
$$
Then $C$ is a direct summand of $W^2\oplus W^3$, which implies $C\in \W$.
\end{proof}

\begin{cor}\label{L2c}
Any epimorphism $A\xrightarrow{\underline a} B$ in $\E/\W$ admits an $\EE$-triangle $A\xrightarrow{a'} B'\to W\dashrightarrow$ such that
\begin{itemize}
\item[(a)] $a'$ is $\W$-monic;
\item[(b)] $\underline a=\underline a'$;
\item[(c)] $W\in \W$.
\end{itemize}
\end{cor}

\begin{lem}\label{L3}
Let $A\xrightarrow{a} B\xrightarrow{b} C'\dashrightarrow$ be an $\EE$-triangle such that $a$ is $\W$-monic. Then the following conditions are equivalent.
\begin{itemize}
\item[(1)] $C'\in \W$.
\item[(2)] $\underline a$ is an epimorphism in $\E/\W$.
\end{itemize}
\end{lem}

\begin{proof}
We first show (2) $\Rightarrow$ (1). By Lemma \ref{L2}, $a\colon A\to B$ admits an $\EE$-triangle
$$A\xrightarrow{\svecv{-a}{w_A}} B\oplus W^1\longrightarrow C\dashrightarrow$$
where $W^1\in\X$ and $C\in \W$. Since $a$ is $\W$-monic, there exists
a morphism $w'\colon B\to W^1$ such that $w'a=w_A$.
Thus we have the following commutative diagram
$$\xymatrix{
A \ar[r]^a \ar@{=}[d] &B \ar[r] \ar[d]^{\svecv{-1}{w'}} &C' \ar[d]^g \ar@{-->}[r] &\\
A\ar[r]^{\svecv{-a}{w_A}\quad} \ar@{=}[d] &B\oplus W^1 \ar[r] \ar[d]^{(-1~0)} &C \ar[d]^h \ar@{-->}[r] &\\
A \ar[r]^a &B \ar[r] &C'\ar@{-->}[r] &
}
$$
of $\EE$-triangles. By \cite[Corollary 3.6]{NP},
we have that $hg$ is an isomorphism. Hence $C'$ is a direct summand of $C$, which implies $C'\in \W$.

Now we show (1) $\Rightarrow$ (2). Let $d:B\to D$ be a morphism such that $\underline {da}=0$. Hence $da$ factors through an object $W\in \W$:
$$\xymatrix{
A \ar[r]^a \ar[d]_{w_1} &B \ar[d]^d\\
W \ar[r]_{w_2} &D
}
$$
Since $a$ is $\W$-monic, there is a morphism $w_3:B\to W$ such that $w_3a=w_1$. Hence $(d-w_2w_3)a=0$, which means there exists a morphism $c:C'\to D$ such that $d=w_2w_3+cb$. Thus $\underline d=0$, whence $\underline a$ is an epimorphism in $\E/\W$.
\end{proof}

\begin{defn}\rm
Denote by ${_\mathfrak{e}}\W$ the of the subcategory consisting of the objects $W\in \W$ which admit $\EE$-triangles
$$A\xrightarrow{a} B\xrightarrow{b} W\dashrightarrow$$
where $a$ is left minimal and $\W$-monic. Such $\EE$-triangles are called \defines{epic-$\EE$-triangles}. Dually we can define
the subcategory $\W_{\mathfrak{m}}$ and the corresponding monic-$\EE$-triangles.
\end{defn}

\begin{rem}\rm
By definition, in an epic-$\EE$-triangles we have $A\notin \W$.
\end{rem}

Let $\C,\D$ be subcategories of $\E$. We denote by $\D\backslash \C$ the additive closure of subcategory consisting of the objects in $\D$ which have no direct summand in $\C$.

\begin{lem}\label{right}
Assume that $\E$ is Krull-Schmidt. Let $W_0$ be an indecomposable object in $\W \backslash {_\mathfrak{e}}\W$. Let
$$A\xrightarrow{a} B\xrightarrow{b} W\dashrightarrow$$
be an epic-$\EE$-triangle. Then any morphism $w:W_0\to W$ factors through $b$.
\end{lem}

\begin{proof}
We have the following commutative diagram
$$\xymatrix{
A\ar[r]^{a_0} \ar@{=}[d] &B_0 \ar[r]^{b_0} \ar[d] &W_0 \ar[d]^{w} \ar@{-->}[r] &   \\
A \ar[r]_a &B \ar[r]_b &W \ar@{-->}[r] &
}
$$
Since $a$ is $\W$-monic, so is $a_1$. If the first row does not split, then $a_1$ is left minimal, which implies that $W_0\in {_\mathfrak{e}}\W$, a contradiction. Hence the first row splits and $w$ factors through $b$.
\end{proof}

\begin{cor}\label{rightc}
Assume that $\E$ is Krull-Schmidt. Let $W_0$ be an object in $\W \backslash {_\mathfrak{e}}\W$. Let
$$A\xrightarrow{a} B\xrightarrow{b} W\dashrightarrow$$
be an $\EE$-triangle where $W\in \W$ and $a$ is $\W$-epic. Then any morphism $w:W_0\to W$ factors through $b$.
\end{cor}

\begin{proof}
For convenience, we can assume that $W_0$ is indecomposable. Since $\E$ is Krull-Schmidt, we can rewrite the $\EE$-triangle:
$$A\xrightarrow{a=\svecv{a'}{0}} B'\oplus W_2\xrightarrow{b=\left(\begin{smallmatrix}b'&0\\0&1\end{smallmatrix}\right)} W_1\oplus W_2\dashrightarrow$$
where $a'$ is left minimal. Hence we can get an epic-$\EE$-triangle
$$A\xrightarrow{a'} B'\xrightarrow{b'} W_1\dashrightarrow$$
By Lemma \ref{right}, for any morphism $w_1:W_0\to W_1$, there exists a morphism $v_1:W_0\to B$ such that $w_1=b'v_1$. Hence for any morphism $W_0\xrightarrow{\svecv{w_1}{w_2}} W_1\oplus W_2$, we have the following commutative diagram:
$$\xymatrix{
&&W\ar[d]^{\svecv{w_1}{w_2}} \ar[ld]_-{\svecv{v_1}{w_2}}\\
A \ar[r]_-{\svecv{a'}{0}} &B'\oplus W_2 \ar[r]_-{\left(\begin{smallmatrix}b'&0\\0&1\end{smallmatrix}\right)} &W_1\oplus W_2 \ar@{-->}[r] &
}
$$
\end{proof}

\begin{lem}\label{sum}
If $\E$ is Krull-Schmidt, then ${_\mathfrak{e}}\W$ and $\W_{\mathfrak{m}}$ are closed under direct sums and direct summands.
\end{lem}

\begin{proof}
Let $W_1,W_2\in {_\mathfrak{e}}\W$. Then $W_i~(i=1,2)$ admits an epic-$\EE$-triangle
$$A_i\xrightarrow{a_i} B_i\xrightarrow{b_i} W_i\dashrightarrow.$$
The direct sum of these two epic-$\EE$-triangles
$$A_1\oplus A_2 \xrightarrow{\left(\begin{smallmatrix}a_1&0\\0&a_2\end{smallmatrix}\right)} B_1\oplus B_2\xrightarrow{\left(\begin{smallmatrix}b_1&0\\0&b_2\end{smallmatrix}\right)} W_1\oplus W_2 \dashrightarrow$$
is still an epic-$\EE$-triangle. Hence $W_1\oplus W_2\in {_\mathfrak{e}}\W$.

Let $W=W^1\oplus W^2\in {_\mathfrak{e}}\W$, where $W^1$ is indecomposable. Then $W$ admits an epic-$\EE$-triangle
$$A\xrightarrow{a} B\xrightarrow{\svecv{b^1}{b^2}} W^1\oplus W^2\dashrightarrow.$$
If $W^1\in \W \backslash {_\mathfrak{e}}\W$, by Lemma \ref{right}, $W^1\xrightarrow{\svecv{1}{0}} W^1\oplus W^2$ factors through $\svecv{b^1}{b^2}$. Hence $W^1$ is a direct summand of $B$. Let $B=W^1\oplus B^1$, the epic-$\EE$-triangle can be rewritten in the following way:
$$A\xrightarrow{a=\svecv{a_1}{a_2}} W^1\oplus B^1 \xrightarrow{\left(\begin{smallmatrix}1&{b_1}'\\0&{b_2}'\end{smallmatrix}\right)} W^1\oplus W^2\dashrightarrow.$$
Then we have the following commutative diagram
$$\xymatrix{
A\ar[r]^-{\svecv{a_1}{a_2}} \ar@{=}[d] &W^1\oplus B^1 \ar[d]^{\left(\begin{smallmatrix}1&{b_1}'\\0&1\end{smallmatrix}\right)}_{\simeq} \ar[r]^{\left(\begin{smallmatrix}1&{b_1}'\\0&{b_2}'\end{smallmatrix}\right)} &W^1\oplus W^2\ar@{=}[d] \ar@{-->}[r] &\\
A\ar[r]_-{\svecv{0}{a_2}}  &W^1\oplus B^1  \ar[r]_{\left(\begin{smallmatrix}1&0\\0&{b_2}'\end{smallmatrix}\right)} &W^1\oplus W^2 \ar@{-->}[r] &
}
$$
which implies that the second row is also an epic-$\EE$-triangle. But ${\svecv{0}{a_2}}$ is not left minimal, a contradiction. Hence $W^1\in {_\mathfrak{e}}\W$.
\end{proof}

For convenience, denote $\EE_{\W\backslash {_\mathfrak{e}}\W}$ by $\EE_{\mathfrak{e}}$ and $\EE^{\W\backslash \W_\mathfrak{m}}$ by $\EE^{\mathfrak{m}}$.

\begin{prop}\label{here}
Assume that $\E$ is Krull-Schmidt. Then $\W$ is hereditary-like in the following extriangulated subcategories:
\begin{itemize}
\item[(1)] $(\E,\EE_{\mathfrak{e}},\mathfrak{s}|_{\EE_{\mathfrak{e}}})$,
\item[(2)] $(\E,\EE^{\mathfrak{m}},\mathfrak{s}|_{\EE_{\mathfrak{m}}})$.
\end{itemize}
\end{prop}

\begin{proof}
For any indecomposable object $A\notin \W$, $A$ admits an $\EE$-triangle $W_2\to W_1\xrightarrow{w} A\overset{\delta}\dashrightarrow$ where $W_1,W_2\in\W$ and $w$ is a minimal right $\W$-approximation. By definition $\delta\in \EE_{\mathfrak{e}}(A,W_2)$. On the other hand, $A$ admits an $\EE$-triangle $A\xrightarrow{w'} W^1\to W^2 \overset{\sigma}\dashrightarrow$ where $W^1,W^2\in\W$ and $w'$ is a minimal left $\W$-approximation. By Lemma \ref{right}, $\sigma\in \EE_{\mathfrak{e}}(W^2,A)$. Hence by Proposition \ref{main1},  $\W$ is hereditary-like in the extriangulated subcategory $(\E,\EE_{\mathfrak{e}},\mathfrak{s}|_{\EE_{\mathfrak{e}}})$. By the dual of Lemma \ref{right}, we can get that $\W$ is hereditary-like in the extriangulated subcategory  $(\E,\EE^{\mathfrak{m}},\mathfrak{s}|_{\EE_{\mathfrak{m}}})$.
\end{proof}



\section{quasi-abelian quotients}


By Lemma \cite[Lemma 3.12]{HH}, if $A\xrightarrow{a} B\xrightarrow{b} C\dashrightarrow$ is an $\EE_{\W}^{\W}$-triangle, then $A\xrightarrow{\underline a} B\xrightarrow{\underline b} C$ is a kernel-cokernel pair in $\E/\W$. Since the ``$\EE$-triangles" in $\E/\W$ are just the images of $\EE_{\W}^{\W}$-triangles, by \cite[Corollary 3.18]{NP}, $\E/\W$ has an exact category structure in the sense of \cite{Bu} and the class of short exact sequences is just the image of the all the $\EE_{\W}^{\W}$-triangles, we denote it by $\mathfrak{S}$. Thus $(\E/\W,\mathfrak{S})$ is an exact category.

\begin{thm}\label{ck}
$(\E/\W,\mathfrak{S})$ is the largest exact category structure on $\E/\W$ in the sense that any kernel-cokernel pair in $\E/\W$ belongs to $\mathfrak{S}$.
\end{thm}

\begin{proof}
As noted above, $(\E/\W,\mathfrak{S})$ is an exact category. Let $A\xrightarrow{\underline a} B\xrightarrow{\underline b} C$ be a kernel-cokernel pair in $\E/\W$. $A$ admits an $\EE$-triangle $A\xrightarrow{w^1} W^1 \to W^2 \dashrightarrow$
with $W^1,W^2\in \W$ and $w^1$ is a left $\W$-approximation. Then we have the following commutative diagram of $\EE$-triangles
$$\xymatrix{
A\ar[r]^{w^1} \ar[d]_a &W^1 \ar[r] \ar[d]^{w} &W^2 \ar@{=}[d] \ar@{-->}[r] &\\
B\ar[r]_{b'} &C' \ar[r] &W^2 \ar@{-->}[r] &.
}
$$
By Lemma \ref{L4}, we can choose morphism $m$ to make an $\EE$-triangle
$$A \xrightarrow{\svecv{-a}{w^1}} B\oplus W^1 \xrightarrow{\svech{b'}{w}} C' \dashrightarrow.$$
Since $w^1$ is a left $\W$-approximation, $\svecv{-a}{w^1}$ is $\W$-monic. By \cite[Lemma 3.12]{HH}, $\underline b':B\to C'$ is the cokernel of $\underline a$. Then we have the following commutative diagram
$$\xymatrix{
A \ar[r]^{\underline a} \ar@{=}[d] &B \ar[r]^{\underline b} \ar@{=}[d] &C \ar[d]^{\simeq}\\
A \ar[r]^{\underline a}  &B \ar[r]^{\underline b'}  &C'.
}
$$
Since $C$ admits an $\EE$-triangle $W_2 \to W_1\xrightarrow{w'} C \dashrightarrow$ where $W_1,W_2\in \W$ and $w'$ is a right $\W$-approximation. Then we can get the following commutative diagram
$$\xymatrix{
W_2 \ar@{=}[r] \ar[d] &W_2 \ar[d]\\
A' \ar[r] \ar[d]_{a'} &W_1 \ar[r] \ar[d]^{w'} &W^2 \ar@{=}[d] \ar@{-->}[r] &\\
B \ar[r]_{b'} \ar@{-->}[d]  &C' \ar[r] \ar@{-->}[d] &W^2\ar@{-->}[r] &.\\
&&
}
$$
By the dual of Lemma \ref{L4}, we can get an $\EE$-triangle
$$A' \xrightarrow{\svecv{a'}{w_1}} B\oplus W_1 \xrightarrow{\svech{-b'}{w'}} C' \dashrightarrow.$$
Since $w'$ is a right $\W$-approximation, $\svech{-b'}{w'}$ is $\W$-epic. Then we get the following commutative diagram
$$\xymatrix{
A \ar[r]^-{\svecv{-a}{w^1}} \ar[d] &B\oplus W^1 \ar[r]^-{\svech{b'}{w}} \ar[d]^{\left(\begin{smallmatrix}-1&0\\0&m_0\end{smallmatrix}\right)} &C' \ar@{=}[d] \ar@{-->}[r]&\\
A' \ar[r]_-{\svecv{a'}{w_1}} &B\oplus W_1 \ar[r]_-{\svech{-b'}{w'}} &C' \ar@{-->}[r]&.
}
$$
By Lemma \ref{L4}, we can replace morphism $\left(\begin{smallmatrix}-1&0\\0&m_0\end{smallmatrix}\right)$ by a morphism $\alpha$ to make the left square a weak push-out. Then $\svecv{-a}{w^1}$ is $\W$-monic implies that $\svecv{a'}{w_1}$ is also $\W$-monic. Hence second row is an $\EE_{\W}^{\W}$-triangle, which implies that $A'\xrightarrow{\underline a'} B\xrightarrow{\underline b'}C'$ is a short exact sequence in $\mathfrak{S}$. Then we have the following commutative diagram
$$\xymatrix{
A \ar[r]^{\underline a} \ar[d]^{\simeq} &B \ar[r]^{\underline b'} \ar@{=}[d] &C' \ar@{=}[d]\\
A' \ar[r]^{\underline a'}  &B \ar[r]^{\underline b'}  &C'.
}
$$
Hence $A\xrightarrow{\underline a} B\xrightarrow{\underline b} C$ is isomorphic to a short exact sequence in $\mathfrak{S}$, which implies that itself is a short exact sequence in $\mathfrak{S}$.
\end{proof}

Now we can show the following theorem.

\begin{thm}\label{main}
$\E/\W$ is quasi-abelian.
\end{thm}

\begin{proof}
Let $\underline i:A\to B$ be a kernel in $\E/\W$. Since by \cite[Theorem 3.9]{HH}, $\E/\W$ is semi-abelian, $\underline i$ admits a kernel-cokernel pair, which also means that it admits a short exact sequence
$$\xymatrix{A \ar[r]^{\underline i} &B \ar[r]^{\underline p} &C}$$
in $(\E/\W,\mathfrak{S})$. For any push-out diagram
$$\xymatrix{
A \ar[r]^{\underline i} \ar[d] &B \ar[d]\\
A' \ar[r]_{\underline {i'}} &B'
}
$$
by the definition of exact category, $\underline {i'}$ admits a short exact sequence $A' \xrightarrow{\underline {i'}} B'\xrightarrow{\underline {p'}} C'$. Hence $\underline {i'}$ is a kernel, which implies that $\underline \h$ is right quasi-abelian. A semi-abelian category is right quasi-abelian if and only if it is left quasi-abelian. Hence $\underline \h$ is quasi-abelian.
\end{proof}

\section{Abelian quotients}

Since $\E/\W$ is semi-abelian, it is abelian if and only if any morphism which is both epic and monic is in fact isomorphic. On the other hand, if a morphism is both a cokernel and a monomorphism, then it is an isomorphism. Hence if any epimorphism in $\E/\W$ is a cokernel, then $\E/\W$ is abelian.

\begin{prop}\label{mo}
Assume that $\E/\W$ is abelian. Consider the following commutative diagram of $\EE$-triangles
$$\xymatrix{
W_2 \ar@{=}[r] \ar[d] &W_2 \ar[d]\\
A \ar[r]^w \ar[d]_{a} &W_1 \ar[r] \ar[d]^{w_1} &W \ar@{=}[d] \ar@{-->}[r] & \\
B \ar[r]_{b} \ar@{-->}[d]  &C \ar[r] \ar@{-->}[d] &W\ar@{-->}[r] &\\
&&
}
$$
with $W, W_1, W_2\in \W$, if $b$ is $\W$-monic and $w_1$ is a right $\W$-approximation, then $w$ is a left $\W$-approximation.
\end{prop}

\begin{proof}
Since $b$ is $\W$-monic, $\underline b$ is an epimorphism. Then $\E/\W$ is abelian implies that $\underline b$ is a cokernel. By the proof of Theorem \ref{ck}, there is a morphism $w':A\to W_1$ such that $\svecv{a}{w'}:A\to B\oplus W_1$ is $\W$-monic. Note that $w'-w$ factors through $a$, hence there is a morphism $b_1:B\to W_1$ such that $b_1a=w'-w$. Since $b$ is $\W$-monic, there is a morphism $c_1:C\to W_1$ such that $b_1=c_1b$. Let $w_0:A\to W_0$ be any morphism with $W_0\in \W$. Then there is a morphism $\svech{f}{g}:B\oplus W_1\to W_0$ such that $w_0=fa+gw'$. Since $b$ is $\W$-monic, there is a morphism $f_0:C\to W_0$ such that $f=f_0b$. Hence $w_0=f_0ba+g(w+c_1ba)=f_0w_1w+gw+gc_1w_1w$, which implies that $w$ is a left $\W$-approximation.
\end{proof}

By Proposition \ref{mo} and Lemma \ref{sum}, we can get the following corollary.

\begin{cor}\label{moc}
Any indecomposable object $W\in {_\mathfrak{e}}\W$ admits an $\EE$-triangle
$$A_0\xrightarrow{a_0} W_0\to W\dashrightarrow$$
where $w$ is a minimal left $\W$-approximation.
\end{cor}

\section{cluster tilting in the subcategories}

In this section, we assume that $\E$ is Krull-Schmidt.

\begin{defn}\rm
Let $(\E,\mathbb{F},\mathfrak{s}|_{\mathbb{F}})$ be a subcategory of $\E$. A subcategory $\M$ is called an \defines{$\mathbb{F}$-cluster tilting subcategory} if the following conditions are satisfied:
\begin{itemize}
\item[(1)] $\M$ is hereditary-like in $(\E,\mathbb{F},\mathfrak{s}|_{\mathbb{F}})$.
\item[(2)] $\mathbb{F}(\M,\M)=0$.
\end{itemize}
\end{defn}

%
%
%
%
%
%
%
%
%


\begin{defn} \rm
Let $\mathcal A$ be any Krull-Schmidt additive category.
\begin{itemize}
\item[(a)] $\mathcal A$ is called \defines{left locally finite} if for any indecomposable object $Y$, there are only finite many indecomposable objects $X$ such that $\Hom_{\mathcal A}(X,Y)\neq 0$.
\item[(b)] $\mathcal A$ is called \defines{right locally finite} if for any indecomposable object $Y$, there are only finite many indecomposable objects $Z$ such that $\Hom_{\mathcal A}(Y,Z)\neq 0$.
\item[(c)] $\mathcal A$ is called locally finite if it is both left and right locally finite.
\end{itemize}
\end{defn}


\begin{lem}\label{e}
Assume that $\E/\W$ is left locally finite. If $\E/\W$ is abelian, then any indecomposable object $W\in {_\mathfrak{e}}\W$ admits an $\EE$-triangle
$$A_j\xrightarrow{a_j} W_j\to W\dashrightarrow$$
where $a_j$ is a minimal left $\W$-approximation and $W_j\in \W\backslash {_\mathfrak{e}}\W$.
\end{lem}

\begin{proof}
By Corollary \ref{moc}, $W\in {_\mathfrak{e}}\W$ admits an $\EE$-triangle
$$A_0\xrightarrow{a_0} W_0\to W\dashrightarrow$$
where $a_0$ is a minimal left $\W$-approximation. If $W_0\notin \W\backslash {_\mathfrak{e}}\W$, then $W_0$ admits an $\EE$-triangle
$$B_1\xrightarrow{b_1} W_1'\to W_0\dashrightarrow$$
where $b_1$ is a minimal left $\W$-approximation. Hence we have the following commutative diagram
$$\xymatrix{
B_1 \ar@{=}[r] \ar[d] &B_1 \ar[d]^{b_1}\\
A_1' \ar[d]^{c_1'} \ar[r]^{a_1'} &W_1' \ar[r] \ar[d] &W \ar@{=}[d] \ar@{-->}[r] &\\
A_0 \ar[r]_{a_0} \ar@{-->}[d] &W_0 \ar[r] \ar@{-->}[d] &W \ar@{-->}[r] &\\
&&
}
$$
where $a_1'$ is a left $\W$-approximation. Let $A_1'=A_1\oplus W^1$, where $A_1$ has no direct summand in $\W$ and $W^1\in \W$. Then we have following commutative diagram
$$\xymatrix{
A_1 \ar[d]^{\svecv{1}{0}} \ar[r]^{a_1} &W_1 \ar[r] \ar[d]^{\svecv{1}{0}} &W \ar@{=}[d] \ar@{-->}[r] &\\
A_1\oplus W^1 \ar[d]^{c_1'=\svech{c_1}{w'}} \ar[r]^-{\left(\begin{smallmatrix}a_1&0\\0&1\end{smallmatrix}\right)} &W_1\oplus W^1 \ar[d] \ar[r] &W \ar@{=}[d] \ar@{-->}[r] &\\
A_0 \ar[r]_{a_0}  &W_0 \ar[r]  &W \ar@{-->}[r] &
}
$$
where $a_1$ is a left $\W$-approximation. Since $W$ is indecomposable, $a_1$ is left minimal. Moreover, we can get that $A_1$ is also indecomposable. If $W_i\notin \W\backslash {_\mathfrak{e}}\W(i\geq 1)$, then by the same method, we can get the following commutative diagram
$$\xymatrix{
A_{i+1} \ar[d] \ar[r]^{a_{i+1}} &W_{i+1} \ar[r] \ar[d] &W \ar@{=}[d] \ar@{-->}[r] &\\
A_i \ar[r]_{a_i}  &W_i \ar[r]  &W \ar@{-->}[r] &
}
$$
where $A_{i+1}$ is indecomposable and $a_{i+1}$ is a minimal left $\W$-approximation.
Since $\E/\W$ is left locally finite, if any $W_i\notin \W\backslash {_\mathfrak{e}}\W(i\geq 1)$, there must be $j,k$ such that $0\leq j<k$ such that $A_j\simeq A_k$. Otherwise, there exists $l>0$ such that $\Hom_{\E/\W}(A_l,A_0)=0$. Consequently, the composition $c_l\cdots c_1$ factors through $\W$, which implies that it factors through $a_l$. But this means $W$ is a direct summand of $W_0$ and $A_0\xrightarrow{a_0} W_0\to W\dashrightarrow$ splits, a contradiction.

Now we can assume that $j=0$ and $A_k=A_0$. Hence we have the following commutative diagram
$$\xymatrix{
A_0 \ar[d]^{c^0} \ar[r]^{a_0} &W_0 \ar[r] \ar[d] &W \ar@{=}[d] \ar@{-->}[r] &\\
A_1 \ar[d]^{c_1} \ar[r]^{a_1} &W_1 \ar[r] \ar[d] &W \ar@{=}[d] \ar@{-->}[r] &\\
A_0 \ar[r]_{a_0}  &W_0 \ar[r]  &W \ar@{-->}[r] &
}
$$
Then there exists a morphism $w_0:W_0\to A_0$ such that $1-c_1c_0=w_0a_0$. Since $A_0\notin \W$, $w_0a_0$ can not be invertible. Then $A_0$ is indecomposable implies that $c_1c_0=1-w_0a_0$ is invertible. Since $A_1$ is also indecomposable, $c_1,c_0$ are isomorphisms. Then $c_1'$ is a retraction, which implies that $B_1\simeq W^1$. Since $b_1$ is left minimal, this is a contradiction. Hence there is an $\EE$-triangle
$$A_j\xrightarrow{a_j} W_j\to W\dashrightarrow$$
where $a_j$ is a minimal left $\W$-approximation, and $W_j\in \W\backslash {_\mathfrak{e}}\W$.
\end{proof}

Let $\underline {\mathcal P}$ be the subcategory of $\E$ such that any non-zero indecomposable object in $A\in \underline {\mathcal P}$ admits an $\EE$-triangle
$$A\xrightarrow{a'} W'\to W\dashrightarrow$$
where $a'$ is a minimal left $\W$-approximation, $W_j\in \W\backslash {_\mathfrak{e}}\W$, $W\in {_\mathfrak{e}}\W$ is indecomposable.
We have the following observation.

\begin{prop}
Assume that $\E/\W$ is left locally finite. If $\E/\W$ is abelian, then it has enough projectives $\underline {\mathcal P}$.
\end{prop}

\begin{proof}
We first show that $\underline {\mathcal P}$ is a subcategory of projectives of $\E/\W$.

Let $\underline b:B\to C$ be an epimorphism in $\E/\W$. Then $\overline b$ admits an $\EE$-triangle
$$B\xrightarrow{b} C\xrightarrow{w_0} W_0\dashrightarrow$$
where $b$ in $\W$-monic and $W_0\in \W$. Let $A\in \underline {\mathcal P}$. Then it admits an $\EE$-triangle
$$A\xrightarrow{a'} W'\to W\dashrightarrow$$
where $a'$ is a minimal left $\W$-approximation, $W'\in \W\backslash {_\mathfrak{e}}\W$ and $W\in {_\mathfrak{e}}\W$.
Let $\overline a:A\to C$ be any morphism in $\E/\W$. Since $a'$ is a left $\W$-approximation, there is a morphism $w':W'\to W_0$ such that $w_0a=w'a'$:
$$\xymatrix{
&A\ar[r]^{a'} \ar[d]^{a} &W'\ar[r] \ar@{.>}[d]^{w'} &W\ar@{-->}[r] &\\
B\ar[r]_{b} &C\ar[r]_{w_0} &W_0\ar@{-->}[r] &
}
$$
By Corollary \ref{rightc}, there is a morphism $v':W'\to C$ such that $w_0v'=w'$. Hence $w_0(a-v'a')=0$, which implies that there is a morphism $u:A\to B$ such that $bu=a-v'a'$. Thus $\underline {bu}=\underline a$, which implies $A$ is projective in $\E/\W$.

Now we show that $\underline {\mathcal P}$ is a subcategory of enough projectives in $\E/\W$.

Let $A_0$ be any indecomposable object in $\E/\W$. $A_0$ admits an $\EE$-triangle
$$A_0\xrightarrow{a_0'} W_0'\to W\dashrightarrow$$
where $a'$ is a minimal left $\W$-approximation and $W\in {_\mathfrak{e}}\W$. If $W_0'\in \W\backslash {_\mathfrak{e}}\W$, then $A_0\in \underline {\mathcal P}$. If $W_0'\notin \W\backslash {_\mathfrak{e}}\W$, by the proof of Lemma \ref{e}, $W$ admits an $\EE$-triangle
$$A\xrightarrow{a'} W'\to W\dashrightarrow$$
where $a'$ is a minimal left $\W$-approximation, $W'\in \W\backslash {_\mathfrak{e}}\W$ and $W\in {_\mathfrak{e}}\W$. Moreover, we have the following commutative diagram
$$\xymatrix{
A\ar[r]^{a'} \ar[d]_{f} &W'\ar[r] \ar[d] &W\ar@{=}[d] \ar@{-->}[r] &\\
A_0\ar[r]^{a_0'}  &W_0'\ar[r]  &W \ar@{-->}[r] &
}
$$
where $\underline f$ is an epimorphism in $\E/\W$. Since $A\in \underline {\mathcal P}$, $\underline {\mathcal P}$ is a subcategory of enough projectives in $\E/\W$.
\end{proof}

By this proposition, we can get the following result.

\begin{prop}
Assume that $\E/\W$ is left locally finite. If $\E/\W$ is abelian, then $\E/\W\simeq \mod \underline {\mathcal P}$.
\end{prop}

Now we show the following theorem.

\begin{thm}\label{thm3}
Assume that $\E/\W$ is left locally finite. Then $\E/\W$ is abelian if and only if $\W$ is an $\EE_{\mathfrak{e}}$-cluster tilting subcategory in extriangulated subcategory $(\E,\EE_{\mathfrak{e}},\mathfrak{s}|_{\EE_{\mathfrak{e}}})$.
\end{thm}

\begin{proof}

We show the ``only if" part. By Proposition \ref{here}, $\W$ is hereditary-like in $(\E,\EE_{\mathfrak{e}},\mathfrak{s}|_{\EE_{\mathfrak{e}}})$. Let $W,W'$ be any indecomposable objects in $\W$. If $W\notin {_\mathfrak{e}}\W$, then by definition $\EE_{\mathfrak{e}}(W,W')=0$. If $W\in {_\mathfrak{e}}\W$, then by Lemma \ref{e}, $W$ admits an $\EE$-triangle
$$A_0\xrightarrow{a_0} W_0\to W\dashrightarrow$$
where $a_0$ is a minimal left $\W$-approximation and $W_0\in \W\backslash {_\mathfrak{e}}\W$. Let $\delta \in \EE_{\mathfrak{e}}(W,W')$. It admits an $\EE$-triangle
$$W'\to U\to W\overset{\delta}\dashrightarrow.$$
Then we have the following commutative diagram
$$\xymatrix{
&W' \ar[d]^-{\svecv{1}{0}} \ar@{=}[r] &W' \ar[d]\\
A_0 \ar[r]^-{\alpha} \ar@{=}[d] &W'\oplus W_0 \ar[r] \ar[d]^-{\svech{0}{1}} &U \ar[d]^{u} \ar@{-->}[r] &\\
A_0 \ar[r]_{a_0} &W_0 \ar[r] \ar@{-->}[d] &W \ar@{-->}[r] \ar@{-->}[d]^{\delta} &\\
&&
}
$$
of $\EE$-triangles. Since $a_0$ is a left $\W$-approximation, so is $\alpha$. We can get the following commutative diagram
$$\xymatrix{
A_0 \ar[r]^{a_0} \ar@{=}[d] &W_0 \ar[r] \ar[d] &W \ar[d]^{u'} \ar@{-->}[r] &\\
A_0 \ar[r]^-{\alpha} \ar@{=}[d] &W'\oplus W_0 \ar[r] \ar[d]^-{\svech{0}{1}} &U \ar[d]^u \ar@{-->}[r] &\\
A_0 \ar[r]_{a_0} &W_0 \ar[r]  &W \ar@{-->}[r] &
}
$$
which implies that $uu'$ is an isomorphism. Hence $u$ is a retraction, which means $\delta=0$.
\end{proof}

Dually, we can get the following result.

\begin{thm}
Assume that $\E/\W$ is right locally finite. Then $\E/\W$ is abelian if and only if $\W$ is an $\EE^{\mathfrak{m}}$-cluster tilting subcategory in extriangulated subcategory $(\E,\EE^{\mathfrak{m}},\mathfrak{s}|_{\EE^{\mathfrak{m}}})$.
\end{thm}

\section{Some examples}

In this section, we provide several examples to illustrate our main results.

\begin{exm} \label{examp:infinite-repr type} \rm
In this example, we provide an representation-infinite $\kk$-algebra over an algebraically closed field $\kk$ for Theorem \ref{thm3}.
Take $A=\kk\Q/\I$ given by the bound quiver
\begin{center}
\begin{tikzpicture}[scale=1.45]
\draw (-1,0) node{1} (1,0) node{2} (0,1.3) node{3};
\draw[line width=1pt][->] (-0.8,-0.2) to[out=-30, in=210] (0.8,-0.2);
\draw[line width=1pt][->] (-0.8,-0.0) --(0.8,-0.0);
\draw[line width=1pt][->] (1,0.2) -- (0.1,1.1);
\draw[line width=1pt][->] (-0.1,1.1) -- (-1,0.2);
\draw (0,0) node[above]{$a'$} (0,-0.45) node[below]{$a$}
      (0.6,0.6) node[above right]{$b$}
      (-0.6,0.6) node[above left]{$c$};
\draw[blue][dashed] ( 0.5,-0.42) arc(-135:130:0.7);
\draw[blue][dashed] (-0.5,-0.42) arc(-45:-310:0.7);
\draw[blue][dashed] (-0.45,0.75) -- (0.45,0.75);
\end{tikzpicture}

$\I = \langle ab, bc, ca\rangle$.
\end{center}
In this example, $A$ is a gentle algebra (cf. \cite{AS1987} or \cite[etc]{AG2008}).
Thus, all indecomposable objects in $\modcat A$ are either string modules or band modules by using \cite{WW1985,BR1987}
(the author of other references, see for example \cite{AG2008,OPS2018,APS2019},
also provided definitions for string modules and band modules, which readers can refer to on their own).
Furthermore, we can divide all indecomposable right $A$-modules to three cases:
\begin{itemize}
  \item[(1)] simple module $S(3)$;
  \item[(2)] string modules whose string does note cross the vertex $3$;
  \item[(3)] band modules $B(n,\lambda)=$
\begin{tikzpicture}[baseline=-0.5] \tiny
\draw (-1,0) node{$\kk^{\oplus n}$} (1,0) node{$\kk^{\oplus n}$} (0,1.3) node{0};
\draw [->] (-0.8,-0.2) to[out=-30, in=210] (0.8,-0.2);
\draw [->] (-0.8,-0.0) --(0.8,-0.0);
\draw [->] (1,0.2) -- (0.1,1.1);
\draw [->] (-0.1,1.1) -- (-1,0.2);
\draw (0,0) node[above]{$1$} (0,-0.45) node[below]{$\bfJ_{n}(\lambda\ne 0)$}
      (0.6,0.6) node[above right]{$0$}
      (-0.6,0.6) node[above left]{$0$};
\end{tikzpicture},
where $\bfJ_n(\lambda)$ is the $n\times n$ Jordan block with eigenvalue $\lambda$.
\end{itemize}
Then for the category $\modcat A$ and its full subcategory $\catW := \langle M \in \obj(\modcat A) \mid Me_1 \ne 0 \text{~or~} Me_3 \ne 0\rangle$
($e_1$ is the idempotent corresponded by the vertex $1$), we have
\[\obj(\modcat A)\backslash\obj(\catW) = \{S(3)_A\}. \]
Moreover, since the projective resolution $\mathbf{P}(S(3))$ of $S(3)_A$ is
\[ \cdots \To{}
   P(1)
     = \left(
       \begin{smallmatrix}
        & 1 & \\
        2 && 2 \\
        && 3
       \end{smallmatrix}
       \right)_A
  \To{h}
  P(3)
    = \left(
       \begin{smallmatrix}
        3 \\ 1 \\ 2\\ 3
       \end{smallmatrix}
       \right)_A
  \To{}
  S(3) = (3)_A
  \To{}
  0,
 \]
we have the following exact sequence
\[ 0 \oT{}
  (1)_{A^{\op}}
  \oT{}
  P(1)_{A^{\op}}
  \oT{h^*}_{(=\Hom_A(h,A))}
  P(3)_{A^{\op}}
    = \left(
       \begin{smallmatrix}
        3 \\ 2 \\ 1\\ 3
       \end{smallmatrix}
       \right)_{A^{\op}}
  \oT{}
  \Hom_A((3)_A, A) \oT{} 0 \]
which is obtained by $\Hom_A(\mathbf{P}(S(3)), A)$. Here, the bound quiver of $A^{\op}$ is obtained by rotating the direction of all arrows of $\Q$.
It follows that $\tau S(3)_A = D((1)_{A^{\op}}) \cong (1)_A = S(1)_A \in \catW$.
One can compute that $\tau^{-1}S(3)_A \cong S(2)_A$ by a dual way.
Furthermore, if a homomorphism $S(3) \to S(3)$ in $\modcat A~/\catW$ is not irreducible,
then it factors through some object in $\obj(\catW)$. Thus, we have
\[ \modcat A~/\catW = \langle S(3)_A \rangle \cong \modcat \kk \]
is Abelian.
The simple module $S(3)_A$ has a $\catW$-resolution
\begin{align}\label{ses1}
  0 \To{} W_1=(1)_A \To{} W_0 = (^3_1)_A \To{w_0} S(3)_A \To{} 0
\end{align}
and a $\catW$-coresolution
\begin{align}\label{ses2}
  0 \To{} (3)_A \To{w^0} W^0 = (^2_3)_A \To{} W^1=S(2)_A \To{} 0.
\end{align}
For arbitrary indecomposable right $A$-module $M_A \in \obj(\catW)$ with $\Hom_A(M, S(3)_A) \ne 0$,
we have any $f: M \to S(3)$ in $\Hom_A(M, S(3)_A)$ is epimorphic,
and $Me_3 \ne 0$. It follows that
\[ M \in
 \left\{
 \left(
   \begin{smallmatrix}
    3 \\ 1
   \end{smallmatrix}
 \right)_A,
  \left(
   \begin{smallmatrix}
    3 \\ 1 \\ 2
   \end{smallmatrix}
 \right)_A,
  \left(
   \begin{smallmatrix}
    3 \\ 1 \\ 2\\ 3
   \end{smallmatrix}
 \right)_A,
  \left(
   \begin{smallmatrix}
    3 &&& \\
    1 &&1 && \cdots\\
    & 2 &&2 & \cdots
   \end{smallmatrix}
 \right)_A
 \right\} \]
up to isomorphism. Thus $f$ must have a decomposition
\[ f=f_1f_2: \xymatrix{ M \ar@{->>}[r]^{f_2} & ({^3_1})_A \ar@{->>}[r]^{f_1} & (3)_A }. \]
This composition admits that $w_0$ is a right $\catW$-approximation since $\dim_{\kk}\Hom_A(({^3_1})_A, (3)_A)=1$.
Dually, $w^0$ is also a left $\catW$-approximation. Thus, $\catW$ is hereditary.

For any $N$ with $\Hom_{\kk}(W_1,N) \ne 0$, since $W_1 = S(1)_A$ is simple, we have $N$ has a submodule which is isomorphic to $S(1)_A$, and then $N = (_1^3)_A$ up to isomorphism.
Then the push-out of
\begin{tikzpicture}[baseline=0.5]
\draw (0,0) node{$\xymatrix{ (1)_A \ar[r] \ar[d] & (^3_1)_A \\ (^3_1)_A & }$};
\end{tikzpicture},
i.e., the following commutative diagram
\[ \xymatrix@C=1.5cm{
 0
  \ar[r]
& W_1
  \ar[r]
  \ar[d]
& W_0
  \ar[r]
  \ar[d]
& (3)_A
  \ar[r]
  \ar@{=}[d]
& 0
\\
 0
  \ar[r]
& (^3_1)_A
  \ar[r]^{[^1_0]}
& (^3_1)_A \oplus (3)_A
  \ar[r]_{[0\ 1]}
& (3)_A
  \ar[r]
& 0,
} \]
given by (\ref{ses1}) provides only one trivial extension
$0 \To{} (^3_1)_A \To{[^1_0]} (^3_1)_A \oplus (3)_A \To{[0\  1]} (3)_A \To{} 0$.
We can consider (\ref{ses2}) by similar way, and so we have $\bbF = \widehat{\{(\ref{ses1}), (\ref{ses2})\}}$.
In this case, $(\catW,\FF,s|_{\FF})$ is a cluster tilting subcategory of $(\catE:=\modcat A,\FF,s|_{\FF})$.
\end{exm}

\begin{exm} \rm
Let $A$ be the path algebra given by the infinite quiver
\[ \Q= \xymatrix{ \cdots \ar[r]^{a_3} & 3^+ \ar[r]^{a_2} & 2^+ \ar[r]^{a_1} & 1 &
 2^- \ar[l]_{b_1} & 3^- \ar[l]_{b_2} & \ar[l]_{b_3} \cdots }, \]
then it is an finite-dimensional path algebra. Furthermore, it is a locally gentle algebra.
Any finitely generated indecomposable right $A$-module in $\modcat A$ is a string module corresponding to a string on $\Q$,
then all non-zero indecomposable modules in $\modcat A$ is divided to three classes as follows:
\begin{itemize}
  \item[(1)] a string module corresponding to a path of the form $a_n \cdots a_{u+1}a_u$ ($n\>= u$)
    or a simple module corresponding to a vertex in $\{2^+,3^+,\cdots\}$,
    and, if $u=1$, then we write it as $[n^+,1,\times]$;
  \item[(2)] a string module corresponding to a path of the form $b_m\cdots b_{v+1}b_v$ ($m\>= v$)
    or a simple module corresponding to a vertex in $\{2^-,3^-,\cdots\}$,
    and, if $v=1$, then we write it as $[\times,1,m^-]$;
  \item[(3)] a string module corresponding to a string of the form $a_n^{-1}\cdots a_1^{-1}b_1\cdots b_m$
  ($= b_m^{-1}\cdots b_1^{-1}a_1\cdots a_n$, $m,n\>=1$), we write it as $[(n+1)^+,1,(m+1)^-]$.
\end{itemize}
Thus, we naturally denote $S(1)$ $[\times,1,\times]$ in this example.
\begin{figure}[htbp]
  \centering
  \includegraphics[width=16cm]{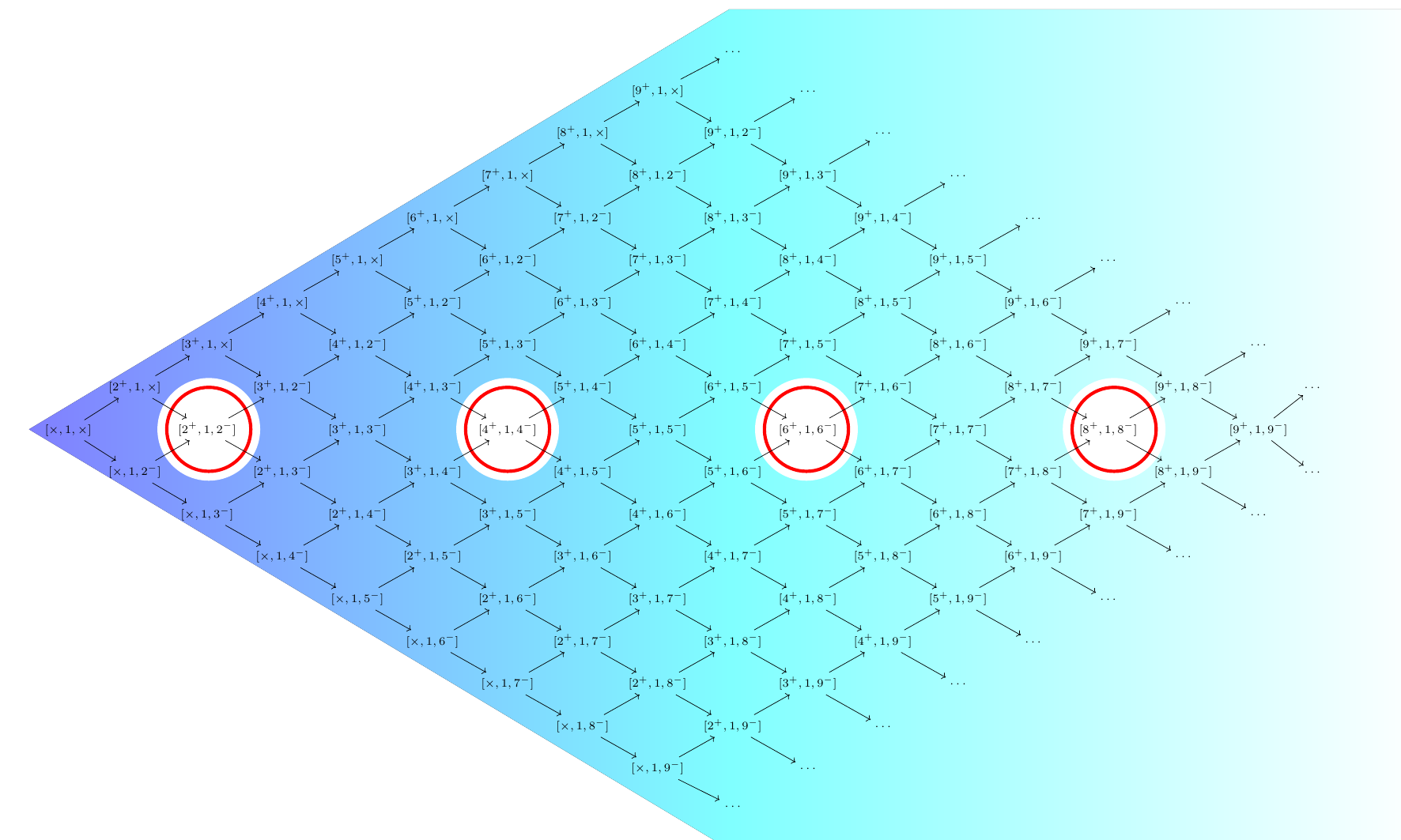}\\
  \caption{A connected component of Auslander--Reiten quiver of $A=\kk\Q$}
  \label{fig:exp20250827}
\end{figure}
Figure \ref{fig:exp20250827} shows a connected component of Auslander--Reiten quiver of $A$.

Take $\catE = \modcat A$ and $\catW = \langle W\in\ind(\modcat A) \mid W\not\cong [(2t)^+,1,(2t)^-] ~ (t\in\NN^+) \rangle$
(see the shadow in \checkLYZ{Figure \ref{fig:exp20250827}},
then for any indecomposable module $X=[(2t)^+,1,(2t)^-]$ (see the vertex marked by ``~{\color{red}$\pmb{\bigcirc}$}'' in \checkLYZ{Figure \ref{fig:exp20250827}}, it has a $\catW$-resolution
\begin{align} \label{ses:1-0829}
  0 \To{} [(2t-1)^+,1,(2t-1)^-] \To{}
{\begin{smallmatrix}
[(2t)^+,1,(2t-1)^-] \hspace{.85cm} \\ \oplus \\ \hspace{.85cm} [(2t-1)^+,1,(2t)^-] \\
\end{smallmatrix}}
 \To{f} X
 \To{} 0 {\color{white}.}
\end{align}
and a $\catW$-coresolution
\begin{align} \label{ses:2-0829}
   0 \To{} X \To{g}
 {\begin{smallmatrix}
[(2t)^+,1,(2t+1)^-] \hspace{.85cm} \\ \oplus \\ \hspace{.85cm} [(2t+1)^+,1,(2t)^-] \\
\end{smallmatrix}}
 \To{} [(2t+1)^+,1,(2t+1)^-]
 \To{} 0.
\end{align}
Then $\FF_0$ is generated by all extensions as above.
We use $\wp_{v_1,v_2}$ to represent the path from $v_1$ to $v_2$ in the quiver $\Q$.
For any indecomposable module $M$ with $\Hom_A(M,[(2t)^+,1,(2t)^-])\ne 0$,
we have $M\cong [x^+,1,y^-]$, where $x,y \in \NN_{\leqslant 2t}$, $x+y<4t$, and if $x=0$ (resp., $y=0$), we take $0^+$ (resp., $0^-$) is $\times$.
In this case, any homomorphism $h: M \to [(2t)^+,1,(2t)^-]$ in $\Hom_A(M,[(2t)^+,1,(2t)^-])$ is induced by the pair $(\wp_{(2t)^+,x^+}, \wp_{(2t)^-,y^-})$ of paths $\wp_{(2t)^+,x^+} = a_{2t-1}\cdots a_{x+1}a_x$ and $\wp_{(2t)^-,y^-} = b_{2t-1}\cdots b_{y+1}b_y$, naturally.
Thus, we have \[h = k_1\wp_{(2t)^+,x^+} + k_2\wp_{(2t)^-,y^-} ~ (k_1,k_2\in\kk)\]
up to isomorphism.
On the other hand, we have $f = [ a_{2t} \ b_{2t} ]$ up to isomorphism. Thus, for the case of $x,y<2t$,
we have the homomorphism
\[ \varphi = \left[\begin{matrix}
  k_1\wp_{(2t-1)^+,x^+} \\ k_2\wp_{(2t-1)^-,y^-}
\end{matrix}\right] : M\cong [x^+,1,y^-] \To{} [(2t)^+,1,(2t-1)^-] \oplus [(2t-1)^+,1,(2t)^-] \]
satisfying
\begin{align*}
  f  \varphi
& = [ b_{2t} \ a_{2t} ]
 \left[\begin{matrix}
  k_1\wp_{(2t-1)^+,x^+} \\ k_2\wp_{(2t-1)^-,y^-}
\end{matrix}\right]
\\
& = k_1a_{2t}\wp_{(2t-1)^+,x^+} + k_2b_{2t}\wp_{(2t-1)^-,y^-} \\
& = k_1\wp_{(2t)^+,x^+} + k_2\wp_{(2t)^-,y^-} = h.
\end{align*}
It follows that the homomorphism $f$ given in (\ref{ses:1-0829}) is a left $\catW$-approximation.
We can prove that $g$ given in (\ref{ses:2-0829}) is a right $\catW$-approximation by using a dual way, then $\catW$ is hereditary.

On the other hand, for any indecomposable module $Y$ with $\Hom_A([(2t-1)^+,1,(2t-1)^-], Y) \ne 0$,
$Y$ is isomorphic to either $M(\wp_{(2t-1)^+, u^+})$ ($u \geqslant 2$) or $M(\wp_{(2t-1)^-, v^-})$ ($v \geqslant 2$).
Here, for any path $\wp$ in $\Q$, $M(\wp)$ is the string module corresponded by $\wp$.
Consider the case for $Y\cong M(\wp_{(2t-1)^+, u^+})$, we have the following push-out
\[ \xymatrix@C=2cm{
 0
  \ar[r]
& \begin{smallmatrix} [(2t-1)^+,1,(2t-1)^-] \end{smallmatrix}
  \ar[r]^{
  \left[\begin{smallmatrix}
   - \emb_1 \\ \emb_2
  \end{smallmatrix}\right]
  }
  \ar[d]_{\text{a homomorphism} \atop \text{induced by }\wp_{u^+,1}}
& {\begin{smallmatrix}
[(2t)^+,1,(2t-1)^-] \hspace{.5cm} \\ \oplus \\ \hspace{.5cm} [(2t-1)^+,1,(2t)^-] \\
\end{smallmatrix}}
  \ar[r]^{\hspace{1cm}f = [ a_{2t} \ b_{2t} ]}
  \ar[d]_{
  \left[\begin{smallmatrix}
   -\epi_1 & \epi_2 \\
   a_{2t} & b_{2t}
  \end{smallmatrix}\right]
  }
& X
  \ar[r]
  \ar@{=}[d]
& 0
\\
 0
  \ar[r]
& M(\wp_{(2t-1)^+, u^+})
  \ar[r]_{[^1_0]}
& {\begin{smallmatrix} M(\wp_{(2t-1)^+, u^+})\oplus X \end{smallmatrix}}
  \ar[r]_{[0\ 1]}
& X
  \ar[r]
& 0
} \]
\begin{center}
  (where $\emb_1$ and $\emb_2$ are canonical embedding,

  $\epi_1$ and $\epi_2$ are canonical epimorphism, and

  $[ a_{2t} \ b_{2t} ] \left[\begin{smallmatrix}
   -\emb_1 \\ \emb_2
  \end{smallmatrix}\right] = -a_{2t}\emb_1 + b_{2t}\emb_2 = 0$)
\end{center}
for (\ref{ses:1-0829}), i.e., the second row is always split.
One can consider the case for $Y\cong M(\wp_{(2t-1)^-, v^-})$ by a dual way, and any pullback given by (\ref{ses:2-0829}) is similar.
Therefore, we have $\FF_1=\FF_0$, and it follows that $(\catW,\FF,s|_{\FF})$ is a cluster tilting subcategory of $(\catE:=\modcat A,\FF,\mathfrak{s}|_{\FF})$.
\end{exm}

\vspace{2mm}

\hspace{-4mm}\textbf{Data Availability}\hspace{2mm} Data sharing not applicable to this article as no datasets were generated or analysed during
the current study.
\vspace{2mm}

\hspace{-4mm}\textbf{Conflict of Interests}\hspace{2mm} The authors declare that they have no conflicts of interest to this work.

\end{document}